\documentclass[12pt]{article}
\usepackage{color}
\usepackage{amsmath}

\usepackage{enumerate}
\usepackage{extarrows}
\usepackage{geometry}
\usepackage{authblk} 
\usepackage{hyperref} 
\usepackage{color}
\usepackage{ntheorem}
\usepackage{enumitem}

\usepackage[mathscr]{euscript}
\usepackage[utf8]{inputenc}
\def\q{\hfill\rule{1ex}{1ex}}
\def\0{\emptyset}

\def\q{\hfill\rule{1ex}{1ex}}

\newtheorem{theorem}{Theorem}[section]

\newtheorem{lemma}[theorem]{Lemma}
\newtheorem{claim}[theorem]{Claim}

\newtheorem{proposition}[theorem]{Proposition}

\newenvironment{provekplF=1}{{\noindent\it Proof of Theorem  \ref{thm: when beta F=1}}.}{\hfill $\square$\par}
\newenvironment{proof}{{\noindent\it Proof.}}{\hfill $\square$\par}

\usepackage{graphicx}

\usepackage{enumerate}
\usepackage{enumitem}

\usepackage{
	amsmath,			
	amssymb,			
	enumerate,		    
	graphicx,			
	lastpage,			
	multicol,			
	multirow,			
	pifont,			    
}

\usepackage[numbers]{natbib}

\newcommand{\mh}{\mathcal{H}}
\newcommand{\mg}{\mathcal{G}}

\newcounter{cases}
\newcounter{subcases}[cases]

\begin{document}
\title{Counting $t$-wise $L$-intersecting cliques with prescribed intersection sizes}
\author{
    {\small\bf Yiyan Zhan}\thanks{email:  zhanyy24@mails.tsinghua.edu.cn}\quad
    {\small\bf Yichen Wang}\thanks{email:  wangyich22@mails.tsinghua.edu.cn}\quad
    {\small\bf Mei Lu}\thanks{email: lumei@tsinghua.edu.cn}\\
    {\small Department of Mathematical Sciences, Tsinghua University, Beijing 100084, China.}\\
}

\date{}

\maketitle\baselineskip 16.3pt

\begin{abstract}
    Let $r,t\ge 3$ be integers and $L=\{\ell_1,\ell_2,\ldots,\ell_s\}\subseteq [0,r-1]$
a fixed set of integers with $|L|\neq r$ and $\ell_1<\ell_2<\cdots<\ell_s$.
For each integer $n$, let $\Psi_{r}(n,L,t)$ be the maximum
number of $r$-cliques in an $n$-vertex graph whose $r$-cliques, viewed as a family of $r$-subsets of the vertex set, form a $t$-wise $L$-intersecting family. In this paper, we prove that $\Psi_{r}(n,L,t)=o(n^{|L|})$ when the sequence $\ell_1,\ell_2,\ldots,\ell_s,r$ does not form an arithmetic progression, and we give an asymptotic formula for $\Psi_{r}(n,L,t)$ when this sequence does form an arithmetic progression. When $t=2$, our results are exactly   Helliar and Liu's results ‌\cite{helliar2024generalizedturanextensiondezaerdhosfrankl}.
\end{abstract}
\vskip.3cm
\section{Introduction}
The classical Tur\'an theorem~\cite{turan1941external} states that for $n\ge q\ge 2$, the maximum number of edges in an $n$-vertex $K_{q+1}$-free graph is attained by the Tur\'an graph $T(n,q)$,
where $T(n,q)$ is the balanced complete $q$-partite graph on $n$ vertices.
Many variants have since been studied. One broad direction is to determine $\text{ex}(n,H,\mathcal{F})$, where $\text{ex}(n,H,\mathcal{F})$ denotes the maximum number of copies of $H$ in an $n$-vertex $\mathcal{F}$-free graph. Tur\'an's theorem is the case $\text{ex}(n,K_2,\{K_{q+1}\})$, and Alon and Shikhelman \cite{ALON2016146} made significant progress for various pairs $(H,\mathcal{F})$. Motivated by results
from extremal set theory, Helliar and Liu \cite{helliar2024generalizedturanextensiondezaerdhosfrankl} studied a generalized Tur\'an problem.

For a positive integer $n$, let $[n]=\{1,2,\ldots,n\}$ and  $2^{[n]}$  the power set of $[n]$. For  ‌nonnegative integers $a$ and $b$ with $a<b$, let $[a,b]=\{a,a+1,\ldots,b\}$.
For integers $n\ge r\ge 1$, $\binom{[n]}{r}$ denotes the collection of all subsets of $[n]$ of size $r$. We call a  family $\mathcal{F}\subseteq\binom{[n]}{r}$ an ($n$-vertex) $r$-graph, and its elements are called hyperedges.
Given a subset $L\subseteq [0,r-1]$, we say that an $r$-graph $\mathcal{F}$ is $L$-intersecting if $|e_1\cap e_2|\in L$ for any two distinct hyperedges $e_1,e_2\in\mathcal{F}$.
For integers $n\ge r\ge 1$ and a set $L\subseteq[0,r-1]$, let $\Phi_r(n,L)$ denote the maximum size of an $n$-vertex $L$-intersecting $r$-graph.
Determining the value of $\Phi_{r}(n,L)$ is a central topic in extremal set theory and is connected to many classical results. The well-known Erd\H{o}s-Ko-Rado theorem \cite{erdos1961intersection} asserts that for integers $r>t\ge 1$, there existed an integer $n_0(r,t)$ such that $\Phi_{r}(n,\{t,t+1,\ldots,r-1\})\le \binom{n-t}{r-t}$ for all $n\ge n_0(r,t)$. The following theorem of Deza, Erd\H{o}s and Frankl established a general upper bound for $\Phi_{r}(n,L)$. It is tight in some cases, but finding matching lower-bound constructions is difficult in general.
\begin{theorem}[Deza-Erd\H{o}s-Frankl \cite{https://doi.org/10.1112/plms/s3-36.2.369}]\label{deza}
    Let $r\ge 3$ and $n\ge2^r r^3$ be integers, and let $L\subseteq [0,r-1]$. Then
    \begin{equation*}
        \Phi_{r}(n,L)\le\prod_{\ell\in L}\frac{n-\ell}{r-\ell}.
    \end{equation*}
    When $L = \emptyset$, we use the convention $\Phi_{r}(n,\emptyset)=1$.
\end{theorem}

Motivated by the study of $L$-intersecting $r$-graphs, Helliar and Liu \cite{helliar2024generalizedturanextensiondezaerdhosfrankl} introduced a generalized Tur\'an problem.
For $r\geq 3$, given a graph $G$, define its associated $r$-graph $\mh_{G}^{r}=\{S\in\binom{V(G)}{r}: G[S]\cong K_r\}$, where $G[S]$ is the subgraph of $G$ induced by the vertex set $S$. For a set of integers $L\subseteq [0,r-1]$, a graph $G$ is called $(K_r,L)$-intersecting if its associated $r$-graph $\mh_{G}^{r}$ is $L$-intersecting. Denote by $N(K_r,G)$ the number of copies of $K_r$ in $G$. Then $N(K_r,G)=|\mh_{G}^{r}|$. Define
\begin{equation*}
    \Psi_{r}(n,L)=
    \max_{\substack{G\text{ is an }n\text{-vertex}\\(K_r,L)\text{-intersecting graph}}}
    N(K_r,G).
\end{equation*}
This framework includes several classical results.
For example, the celebrated Ruzsa-Szemer\'edi Theorem~\cite{ruzsa1978triple} states that $\Psi_{r}(n,\{0,1\})=n^{2-o(1)}$.
By the definition and Theorem~\ref{deza},
\begin{equation*}
    \Psi_{r}(n,L)\le \Phi_{r}(n,L)\le\prod_{\ell\in L}\frac{n-\ell}{r-\ell}.
\end{equation*}
Helliar and Liu improved this upper bound by a constant factor.
\begin{theorem}[Helliar-Liu \cite{helliar2024generalizedturanextensiondezaerdhosfrankl}]\label{heliar} Suppose that $r\ge 3$ and $n \ge(2r)^{
r+1}$ are integers, and that $L\subseteq
[0, r-1]$ is a subset of size $s \in [2, r-1]$. Then
\begin{equation*}
    \Psi_{r}(n,L)\le \left(1-\frac{1}{3r}\right)\prod_{\ell\in L}\frac{n-\ell}{r-\ell}.
\end{equation*}
\end{theorem}
Later, Zhao and Zhang~\cite{zhao2025countingcliquesprescribedintersection} proved the following sharper result.
\begin{theorem}[Zhao-Zhang \cite{zhao2025countingcliquesprescribedintersection}]\label{thm: zhao}
    Let $r\ge 3$ be an integer and let $L=\{\ell_1,\ell_2,\ldots,\ell_s\}\subseteq [0,r-1]$ be a set of size $s\notin \{1,r\}$ with $\ell_1<\ell_2<\cdots<\ell_s$.
    \begin{enumerate}[label=(\arabic*)]
        \item  If $\ell_1,\ell_2,\ldots,\ell_s,r$ do not form an arithmetic progression, then $\Psi_{r}(n,L)=o(n^s)$. If in addition $r-\ell_s = \ell_s-\ell_{s-1}$, then $\Psi_{r}(n,L)=O(n^{s-1})$.
        \item If $\ell_1,\ell_2,\ldots,\ell_s,r$ form an arithmetic progression with common difference $d$, then, when $n$ is sufficiently large,
        \begin{equation*}
            \begin{aligned}
                \Psi_{r}(n,L)
                &=N\left(K_s,T\left(\left\lfloor\frac{n-\ell_1}{d}\right\rfloor, s\right)\right)
                =(1+o(1))\left(\frac{n-\ell_1}{r-\ell_1}\right)^s.
            \end{aligned}
        \end{equation*}
    \end{enumerate}
\end{theorem}

We next recall a generalized extremal set problem that motivates our extension of Helliar and Liu's problem. For a  family $\mathcal{F}\subseteq 2^{[n]}$ and a set of nonnegative integers $L=\{\ell_1,\ell_2,\ldots,\ell_s\}$, we call $\mathcal{F}$ a $t$-wise $L$-intersecting family if the cardinality of the intersection of any $t$ distinct members of $\mathcal{F}$ belongs to $L$. This problem has been studied extensively.
For example, F\"uredi and Sudakov~\cite{FUREDI2004143} proved the following theorem.
\begin{theorem}[F\"uredi-Sudakov~\cite{FUREDI2004143}]\label{fu1} Let $L$ be a set of nonnegative integers of size $s$, and let $t\ge 3$. For sufficiently large $n$, every $t$-wise $L$-intersecting family $\mathcal{F}$ satisfies
\begin{equation*}
    |\mathcal{F}|\le \frac{t+s-1}{s+1}\binom{n}{s}+\sum_{i=0}^{s-1}\binom{n}{i}.
\end{equation*}
\end{theorem}
Kang, Liu, and Wang~\cite{KANG2011464} gave a nonuniform version of the Deza-Erd\H{o}s-Frankl-type theorem.

\begin{theorem}[Kang-Liu-Wang~\cite{KANG2011464}]\label{thm: t L}
Let $t\geq 3$ and let $L=\{\ell_1,\ell_2,\ldots,\ell_s\}$ be a set of $s$ nonnegative integers with $\ell_1<\ell_2<\cdots<\ell_s$. Let $\mathcal{F}$ be a $t$-wise $L$-intersecting family of subsets of $[n]$. If $|\bigcap_{F\in\mathcal{F}}F|<\ell_1$, then $|\mathcal{F}|=o(n^s)$. If $|\bigcap_{F\in\mathcal{F}}F|\geq\ell_1$ and $n$ is sufficiently large, then $|\mathcal{F}|\leq \frac{t+s-1}{s+1}\binom{n-\ell_1}{s}+\sum_{i=0}^{s-1}\binom{n-\ell_1}{i}$.
\end{theorem}

We now extend this notion to clique families.
A graph $G$ is called $t$-wise $(K_r,L)$-intersecting if its associated $r$-graph $\mh_{G}^{r}$ is $t$-wise $L$-intersecting.
Define
\begin{equation*}
    \Psi_{r}(n,L,t)=
    \max_{\substack{G\text{ is an }n\text{-vertex}\\t\text{-wise }(K_r,L)\text{-intersecting graph}}}
    N(K_r,G).
\end{equation*}
In particular, $\Psi_{r}(n,L)=\Psi_{r}(n,L,2)$. Our main results extend Theorem~\ref{thm: zhao} to the $t$-wise family.
The case in which this sequence does not form an arithmetic progression is treated in Theorem \ref{thm 3}, and the arithmetic progression case is treated in Theorem \ref{thm 4}.
\begin{theorem}\label{thm 3}
Let $t,r\geq 3$ be integers and  $L=\{\ell_1,\ell_2,\ldots,\ell_s\}$ with $0\leq\ell_1<\ell_2<\cdots<\ell_s\leq r-1$ and $s\notin\{1,r\}$. If $\ell_1,\ell_2,\ldots,\ell_s,r$ do not form an arithmetic progression, then
    \begin{equation*}
        \Psi_{r}(n,L,t)=o(n^s).
    \end{equation*}
If, in addition, $r-\ell_s = \ell_s-\ell_{s-1}$, then $\Psi_{r}(n,L,t)=O(n^{s-1})$.
\end{theorem}

\begin{theorem}\label{thm 4}
    Let $t,r\geq 3$ be integers and  $L=\{\ell_1,\ell_2,\ldots,\ell_s\}$ with $0\leq\ell_1<\ell_2<\cdots<\ell_s\leq r-1$ and $s\notin\{1,r\}$. If $\ell_1,\ell_2,\ldots,\ell_s,r$ form an arithmetic progression, then
    \begin{equation*}
        \Psi_{r}(n,L,t)=(1+o(1))\left(\frac{n-\ell_1}{r-\ell_1}\right)^{s}.
    \end{equation*}
\end{theorem}

Thus our results recover, for every fixed $t\ge 3$, the same asymptotic behavior as in the case $t=2$ in Theorem~\ref{thm: zhao}.

This article is organized as follows. In Section \ref{sec: preliminaries}, we present some preliminary results. In Section \ref{sec: proof non-AP}, we prove Theorem \ref{thm 3}. In Section \ref{sec: proof AP}, we prove Theorem \ref{thm 4}. Basically,  the main idea of our proofs   comes from  Zhao and Zhang \cite{zhao2025countingcliquesprescribedintersection}, but we make  changes to generalize the results to $t$-wsie case.


\section{Preliminaries}\label{sec: preliminaries}
For graphs $G_1,G_2$, let  $G_1 \cup G_2$ be the union of vertex-disjoint copy of $G_1,G_2$, and let $G_1 + G_2$ be the join of $G_1$ and $G_2$, obtained from $G_1 \cup G_2$ by adding all edges between $G_1$ and $G_2$.
For an integer $m$, let $mG$ denote  the union of $m$ vertex-disjoint copies of $G$.
Given an $r$-graph $\mathcal{H}$ and a vertex $u$, the \emph{link hypergraph} of $u$, denoted by $\mh_{G}^{r}(u)$, is the $(r-1)$-graph whose hyperedge set is $\{S: S\cup \{u\}\in \mathcal{H}\}$.  Let $G$ be a graph and $\mh_{G}^{r}=\{S\in\binom{V(G)}{r}: G[S]\cong K_r\}$ its associated graph. We use  $|\mh_{G}^{r}|$ to represent the order  of $\mh_{G}^{r}$.
For any $v\in V(G)$, let  $d_{\mh_{G}^{r}}(v)$ be the number of subsets in $\mh_{G}^{r}$ contained $v$. Then $d_{\mh_{G}^{r}}(v)=|\mh_{G}^{r}(v)|$.

We say that a collection of sets $A_1,A_2,\ldots,A_p \in \binom{[n]}{r}$ is a \emph{sunflower} with core $C$ if $A_i\cap A_j=C$ for all distinct $i,j\in [p]$ and the set $A_i$
is called the \emph{petal} of the sunflower, $1\le i\le p$.
The core $C$ may be empty, in which case the petals are pairwise disjoint.
The classical sunflower lemma~\ref{lemma: sunflower} provides a sufficient condition for the existence of a sunflower.
\begin{lemma}[Erd\H{o}s-Rado~\cite{https://doi.org/10.1112/jlms/s1-35.1.85}]\label{lemma: sunflower}
    If $\mathcal{J}$ is a collection of subsets, each of size $s$, and $|\mathcal{J}|>s!(k-1)^{s}$, then it contains a sunflower with $k$ petals.
\end{lemma}

For the set of consecutive integers from zero to $s$, we use the following result.
\begin{lemma}\label{lem: [0,s]}
For integers $r,t\geq 3$ and $1\leq s\leq r-2$, we have
\begin{equation*}
    \Psi_{r}(n,[0,s],t)=o(n^{s+1}).
\end{equation*}
\end{lemma}
\begin{proof}
    Let $G$ be an $n$-vertex $t$-wise $(K_r,[0,s])$-intersecting graph.
%
%
%
    %
    Construct an $(s+1)$-hypergraph $\mg$ with $V(\mg)=V(G)$ and $E(\mg)=\{X \in\binom{V(\mg)}{s+1}: G[X]\cong K_{s+1}\}$.
    Let $\mathcal{K}_{r}^{s+1}$ be the complete $(s+1)$-graph on $r$ vertices, that is $\mathcal{K}_{r}^{s+1}=\binom{[r]}{s+1}$.
    Each $K_r$ in $G$ corresponds to a copy of $\mathcal{K}_{r}^{s+1}$ in $\mg$, and each hyperedge of $\mg$ is contained in at most $t-1$ distinct copies of $\mathcal{K}_{r}^{s+1}$.
    Hence the number of copies of $\mathcal{K}_{r}^{s+1}$ in $\mg$ is at most $(t-1)\binom{n}{s+1} = \Theta(n^{s+1})$.

    By the hypergraph removal lemma and $s+1<r$, $\mg$ can be made $\mathcal{K}_{r}^{s+1}$-free by removing $o(n^{s+1})$ hyperedges; these hyperedges correspond to $o(n^{s+1})$ copies of $K_{s+1}$ in $G$.
    Every copy of $K_r$ in $G$ contains one of the removed copies of $K_{s+1}$, and each copy of $K_{s+1}$ is contained in at most $t-1$ copies of $K_r$.
    Hence $G$ contains only $o(n^{s+1})$ copies of $K_r$.
    Therefore, $\Psi_{r}(n,[0,s],t)=o(n^{s+1})$.
\end{proof}

Let $G$ be a graph. We use $\chi(G)$ to denote the chromatic number of $G$. The following theorems will be used in the proof.
\begin{theorem}[Alon-Shikhelman \cite{ALON2016146}]\label{alon pro}
    For any graph $H$ and any integer $q\ge 2$,
    \begin{equation*}
        \text{ex}(n,K_q,H)=\Omega(n^q)
        \quad\text{if and only if}\quad
        \chi(H)>q.
    \end{equation*}
    Furthermore, if $\chi(H)=k>q$, then
    \begin{equation*}
        \text{ex}(n,K_q,H)
        =(1+o(1))\binom{k-1}{q}\left(\frac{n}{k-1}\right)^q.
    \end{equation*}
    \end{theorem}

\section{Proof of Theorem~\ref{thm 3}}\label{sec: proof non-AP}

We begin with the case $s = 2$, proving the theorem first for $L=\{0,\ell\}$ and $r\neq 2\ell$.

\begin{lemma}\label{lem 1}
    Let $t,r\geq 3$ and $1\leq\ell\leq r-1$ be integers. If $r\neq 2\ell$, then
    \begin{equation*}
        \Psi_{r}(n,\{0,\ell\},t)=o(n^2).
    \end{equation*}
\end{lemma}
\begin{proof}
Let $G$ be an $n$-vertex $t$-wise $(K_r,\{0,\ell\})$-intersecting graph, where $r\neq 2\ell$. For any subset $C\subseteq V(G)$ of size $\ell$, let $\mathcal{F}_C\subseteq \mh_{G}^{r}$ be the set of hyperedges containing $C$.
Set
\begin{align*}
        U(G)&=\{u\in V(G):d_{\mh_{G}^{r}}(u)\geq (r-1)!(t-1)^{r-1}+1\},\\
        \mathcal{C}(G)&=\left\{C\in \binom{V(G)}{\ell}:|\mathcal{F}_C|\geq (r-1)!(t-1)^{r-1}+1\right\}.
\end{align*}

We first record several useful properties of $\mathcal{C}(G)$.
\begin{claim}\label{claim: U=C}
    $U(G)=\bigcup_{C\in \mathcal{C}(G)}C$. For any $C\in \mathcal{C}(G)$, $C$ is a core of   a sunflower in $\mh_{G}^{r}$ with $t$ petals.
\end{claim}

 \noindent
    \textbf{Proof of Claim~\ref{claim: U=C} }
  By the definitions, we have $\bigcup_{C\in \mathcal{C}(G)}C\subseteq U(G)$.

  Let $u\in U(G)$.
  Since $\mathcal{H}_G^{r}$ is $t$-wise $\{0,\ell\}$-intersecting, the link hypergraph ${\mh_{G}^{r}}(u)$ of $u$ is a $t$-wise $(\ell-1)$-intersecting $(r-1)$-graph of size at least $(r-1)!(t-1)^{r-1}+1>(r-1)!(t-1)^{r-1}$.
  By Lemma \ref{lemma: sunflower}, ${\mh_{G}^{r}}(u)$ contains a sunflower $\mathcal{S}$ with $t$ petals and core $C$ of size $\ell-1$. Set $\mathcal{S}=\{E_1,\ldots,E_t\}$, where $E_i\in {\mh_{G}^{r}}(u)$. Then $E_i\cap E_j=C$ for all distinct $i,j\in [t]$ and $|C|=\ell-1$. Thus  $\mathcal{S}'=\{E_1\cup\{u\},\ldots,E_t\cup\{u\}\}$
  is a sunflower  in $\mh_{G}^{r}$ with core $C \cup \{u\}$ and $t$ petals.

  For any $E\in \mh_{G}^{r}$ with $u\in E$, we have $|E\cap ((\cup_{i=1}^t(E_i\cup\{u\}))\setminus(E_j\cup\{u\}))|=\ell$ for any $1\le j\le t$ by $\mathcal{H}_G^{r}$ being $t$-wise $\{0,\ell\}$-intersecting.  Hence  $E\cap ((\cup_{i=1}^t(E_i\cup\{u\}))\setminus(E_j\cup\{u\}))=C \cup \{u\}$ for any $1\le j\le t$.
  That is, every hyperedge containing $u$ contains $C \cup \{u\}$.
  Therefore, $C\cup \{u\} \in \mathcal{C}(G)$ by the definition which implies
   $U(G)\subseteq\bigcup_{C\in \mathcal{C}(G)}C$.

   Let $C\in \mathcal{C}(G)$ and $u\in C$. Then the link hypergraph ${\mh_{G}^{r}}(u)$ of $u$ is a $t$-wise $(\ell-1)$-intersecting $(r-1)$-graph of size at least $(r-1)!(t-1)^{r-1}+1>(r-1)!(t-1)^{r-1}$. By the same argument as above and the definition of $\mathcal{F}_C$, the result holds. \q
%
%


\begin{claim}\label{claim: Ccap}
    For any  $C, C'\in \mathcal{C}(G)$, we have $C\cap C'=\emptyset$ or $C=C'$.
\end{claim}

\noindent
    \textbf{Proof of Claim~\ref{claim: Ccap} }Assume $C\cap C'\not=\emptyset$, say
     $u\in C\cap C'$. We will show that $C=C'$. By Claim \ref{claim: U=C},
    there are  sunflowers $\mathcal{S}$ and $\mathcal{S}'$, say $\mathcal{S}=\{E_1,\ldots,E_t\}$ and $\mathcal{S}'=\{E_1',\ldots,E_t'\}$, in $\mh_{G}^{r}$  with $t$ petals whose core is $C$ and $C'$ respectively. Suppose there is  $v\in C\setminus C'$ or $v\in C'\setminus C$, say $v\in C\backslash C'$.
    By the same argument as that of Claim  \ref{claim: U=C}, we have $E_i'\cap (\cup_{j=1}^{t-1}E_{j})=C$ for any $1\le i\le t$. So we have $v\in E_i'$ for any $1\le i\le t$ which implies $v\in C'$, a contradiction.\q
\
\begin{claim}\label{claim: AC}
    For every $A\in \mh_{G}^{r}$ and $C\in \mathcal{C}(G)$, we have either $C\subseteq A$ or $A\cap C=\emptyset$.
\end{claim}

\noindent
    \textbf{Proof of Claim~\ref{claim: AC} }
    Suppose, to the contrary, that $A\cap C \neq \emptyset$ and $C \not\subseteq A$.
    By  Claim \ref{claim: Ccap},
    there is a sunflower $\mathcal{S}$ in $\mh_{G}^{r}$ with core $C$ and $t$ petals.
   Since $A\in \mh_{G}^{r}$ and $A\cap C \neq \emptyset$, by the same argument as that of Claim  \ref{claim: U=C}, we have $C \subseteq A$, a contradiction.\q

We now prove Lemma~\ref{lem 1}.
Start with $G_0=G$.
If there is a vertex $u\in V(G_i)$ with $d_{\mh_{G_i}^{r}}(u)<(r-1)!(t-1)^{r-1}+1$, delete $u$ from $G_i$ to obtain $G_{i+1}$; otherwise, stop.
Suppose that we stop at $G_{k}$.
Then $G_{k}$ is the $t$-wise $(K_{r},L)$-intersecting if $V(G_{k})\not= \emptyset$ and
\begin{equation*}
    |\mh_{G}^{r}|\leq |\mh_{G_k}^{r}|+(r-1)!(t-1)^{r-1}k\leq |\mh_{G_k}^{r}|+(r-1)!(t-1)^{r-1}n.
\end{equation*}
If $V(G_{k})= \emptyset$, then we are done. Now we consider the case that $V(G_{k})\not= \emptyset$. Then every vertex $u\in V(G_{k})$ satisfies $d_{\mh_{G_k}^{r}}(u)\geq (r-1)!(t-1)^{r-1}+1$. So $U(G_{k})=V(G_{k})$.

Since $G_{k}$ is  $t$-wise $(K_{r},L)$-intersecting, Claims~\ref{claim: U=C} and~\ref{claim: Ccap} imply that the sets in $\mathcal{C}(G_k)$ form a partition of $V(G_k)$ into $\ell$-subsets. Hence $\ell \mid |V(G_{k})|$. Let $A\in \mh_{G_k}^{r}$. Then $A\subseteq V(G_{k})$.
By Claim~\ref{claim: AC} and the sets in $\mathcal{C}(G_k)$ forming a partition of $V(G_k)$, we also have $\ell \mid r$. Since $r\neq 2\ell$, $r/\ell\geq 3$.

Now we consider an auxiliary graph $G'$ with $V(G')=\mathcal{C}(G_k)$ and
\begin{equation*}
    E(G')=\left\{\{C,C'\}\in\binom{\mathcal{C}(G_{k})}{2}:~\text{$C\cup C'\subseteq A$ for some $A\in \mh_{G_{k}}^{r}$}\right\}.
\end{equation*}
Thus $|V(G')|=|\mathcal{C}(G_k)|=\frac{|V(G_{k})|}{\ell}$.
\begin{claim}\label{claim: lr}
    We have $|\mh_{G_k}^{r}|=|\mh_{G'}^{\frac{r}{\ell}}|$.
    Moreover, the graph $G'$ is $t$-wise $(K_{r/\ell},\{0, 1\})$-intersecting.
\end{claim}

\noindent
    \textbf{Proof of Claim~\ref{claim: lr} }
    For every hyperedge $A\in \mh_{G_k}^{r}$, the set $\{C\in\mathcal{C}(G_k):~ C\subseteq A\}$ corresponds to an $(r/\ell)$-clique in $G'$ by Claim~\ref{claim: AC} and the definitions. Thus $|\mh_{G_k}^{r}|\leq |\mh_{G'}^{\frac{r}{\ell}}|$.
    Conversely, for a $(r/\ell)$-clique with vertex set $\{C_1,C_2,\ldots,C_{r/\ell}\}$ in $G'$,
    $A=\bigcup_{i=1}^{r/\ell}C_i$ corresponds to an $r$-clique in $G_k$. Thus $|\mh_{G_k}^{r}|\geq |\mh_{G'}^{\frac{r}{\ell}}|$.
    Therefore, $|\mh_{G_k}^{r}|=|\mh_{G'}^{\frac{r}{\ell}}|$.

    Suppose there are $t$ distinct $(r/\ell)$-cliques $K_1,K_2,\ldots,K_t$ in $G'$ such that $|\bigcap_{i=1}^{t}K_i|\geq 2$.
    Write $K_i=\{C_{i,1},C_{i,2},\ldots,C_{i,r/\ell}\}$ for $i=1,2,\ldots,t$.
    Then each $K_i$ corresponds to a $r$-clique $\bigcup_{j=1}^{r/\ell}C_{i,j}$ in $G_k$.
    The corresponding $r$-cliques in $G$ have intersection of size at least $2\ell$, a contradiction.
    Thus $G'$ is $t$-wise $(K_{r/\ell},\{0, 1\})$-intersecting.\q
%

We complete the proof of Lemma~\ref{lem 1}. We have
\begin{equation*}
    \begin{aligned}
        |\mh_{G}^{r}|&\leq |\mh_{G_k}^{r}|+(r-1)!(t-1)^{r-1}n\\
        &=|\mh_{G'}^{r/\ell}|+ (r-1)!(t-1)^{r-1}n\\
        &\leq \Psi_{r/\ell}  \left(\frac{n}{\ell},\{0,1\},t\right)+(r-1)!(t-1)^{r-1}n\\
        &= o(n^2),
    \end{aligned}
\end{equation*}
where the last equality follows from Lemma~\ref{lem: [0,s]}.
\end{proof}

Now we consider the general case $L=\{\ell_1,\ell_2\}\subset [0,r-1]$, where $\ell_1,\ell_2,r$ do not form an arithmetic progression.
If $\ell_1 >0$, then by Theorem \ref{thm: t L}, either $|\mh_{G}^{r}|=o(n^2)$ or every $K_{r}$ in $G$ contains a fixed $K_{\ell_1}$, denoted by $X$.
In the latter case, let $G_X$ be the subgraph of $G$ induced by $\bigcap_{x \in X} N(x)$, where $N(x)=\{u\in V(G):ux\in E(G)\}$. Then $|\mathcal{H}_{G}^{r}| = |\mathcal{H}_{G_X}^{r-\ell_1}|$, and $G_X$ is $t$-wise $(K_{r-\ell_1},\{0,\ell_2-\ell_1\})$-intersecting. This reduces the proof to the case $\ell_1 = 0$, which follows from Lemma \ref{lem 1}.

For the remainder of this section, we prove Theorem~\ref{thm 3} for $s\ge 3$.
We need the following structural theorem.
\begin{theorem}[F\"{u}redi \cite{FUREDI1983129}]\label{furedi stru}
    For integers $p\ge r+1\ge 4$, there exists a positive constant $c=c(p,r)$ such that every $r$-graph $\mathcal{F}$ contains a subhypergraph $\mathcal{F}^*\subset\mathcal{F}$ satisfying the following properties:

    \noindent
    (i) $|\mathcal{F}^*|\ge c|\mathcal{F}|$.

    \noindent
    (ii) The families $\mathcal{I}(F)=\{F\cap F': F'\in \mathcal{F}^*\setminus\{F\}\}$ are isomorphic for all $F\in \mathcal{F}^*$.

    \noindent
    (iii) For any $A\in \mathcal{I}(F)$ with $F\in\mathcal{F}^*$, $A$ is the core of a $p$-sunflower in $\mathcal{F}^*$.

    \noindent
    (iv) For any $F\in \mathcal{F}^*$ and distinct $A_1,A_2\in \mathcal{I}(F)$, $A_1\cap A_2\in \mathcal{I}(F)$.
\end{theorem}

Let $G$ be an extremal $n$-vertex $t$-wise $(K_r,L)$-intersecting graph, and let $\mathcal{F}=\mathcal{H}_{G}^{r}$.
Let $p= \max\{r+1,t\}\ge 4$. By Theorem \ref{furedi stru}, there is a subhypergraph $\mathcal{F}^*\subseteq \mathcal{F}$ satisfying the properties (i)--(iv).
By the property (iii) and $p\ge t\ge 4$, for any $A \in \mathcal{I}(F)$, $A$ is the core of a $p$-sunflower in $\mathcal{F}^*$.
Since $\mathcal{F}^*$ is $t$-wise $L$-intersecting,  $|A| \in L$.
By the definition of $\mathcal{I}(F)$, for all distinct $F_1, F_2 \in \mathcal{F}^*$, we have $|F_1 \cap F_2| \in L$.
In other words, we have
    \begin{equation*}
        \text{(v)~~~~~~~~~~ $\mathcal{F}^*$ is $2$-wise $L$-intersecting.}
    \end{equation*}
    The following proposition directly follows.
    \begin{proposition}\label{linear pro}
        For integers $t,r\ge 3$ and $L=\{\ell_1,\ell_2,\ldots,\ell_s\}$ with $0\le \ell_1<\ell_2<\cdots<\ell_s\le r-1$, there exists a positive constant $C=C(r,t)$ such that
        \begin{equation*}
            \Psi_{r}(n,L,t)\le C\Phi_{r}(n,L).
        \end{equation*}
    \end{proposition}

    Let $A_1,A_2\in \bigcup_{F\in\mathcal{F}^*}\mathcal{I}(F)$ be two distinct sets.
    By the property (iii), $A_i$ is the core of a $p$-sunflower with petals $F_{i,1},\ldots,F_{i,p}$ in $\mathcal{F}^*$ for $i=1,2$.
    Note that $|A_2\setminus A_1|\le r-1$ and each vertex in $A_2\setminus A_1$ can be in at most one petal.
    Then there exists a petal, say $F_{1,1}$, satisfies that $F_{1,1}\cap A_2=A_1\cap A_2$.
    Similarly, since $|F_{1,1}\setminus A_2|\le r<p$, there is a petal, say $F_{2,1}$, satisfies $F_{2,1}\cap F_{1,1} = A_2\cap F_{1,1}=A_1\cap A_2$. By the property (v), we have
    \begin{equation*}
        \text{(vi)~~~~~~~ for two distinct sets } A_1,A_2\in \bigcup_{F\in\mathcal{F}^*}\mathcal{I}(F),\ |A_1\cap A_2|\in L.
    \end{equation*}
    We are ready to prove the following lemma.
    \begin{lemma}\label{key key lem}
        Let $t,r\ge 3$ be integers and  $L=\{\ell_1,\ldots,\ell_s\}\subseteq [0,r-1]$ with $0\le \ell_1<\ell_2<\cdots<\ell_s\le r-1$.
        Let $p = \max\{r+1,t\}$.
        For any $i\in [s-1]$, we have
        \begin{equation*}
            \begin{aligned}
                \Psi_{r}(n,L,t)
                \le c^{-1}\max\bigl\{&
                \Phi_{r}(n,L\setminus\{\ell_i\}),\\
                &\Phi_{\ell_i}(n,\{\ell_1,\ldots,\ell_{i-1}\})
                \Psi_{r-\ell_i}(n-\ell_i,\{0,\ell_{i+1}-\ell_i,\ldots,\ell_s-\ell_i\},t)
                \bigr\},
            \end{aligned}
        \end{equation*}
        where $c=c(p,r)$ is the constant from Theorem~\ref{furedi stru}.
    \end{lemma}

\begin{proof}
Let $G$ be an extremal $n$-vertex $t$-wise $(K_r,L)$-intersecting graph, and let $\mathcal{F}=\mathcal{H}_{G}^{r}$. By Theorem \ref{furedi stru}, there is a subhypergraph $\mathcal{F}^*\subset\mathcal{F}$  satisfying the properties (i)-(vi).
If there is no $A\in\bigcup_{F\in\mathcal{F}^*}\mathcal{I}(F)$ with $|A|=\ell_i$ for some $i\in [s-1]$, then $\mathcal{F}^*$ is an $(L\setminus\{\ell_i\})$-intersecting $r$-graph, and we are done by
\begin{equation*}
    \Psi_{r}(n,L,t)=|\mathcal{F}|\le c^{-1}|\mathcal{F}^*|\le c^{-1}\Phi_{r}(n,L\setminus\{\ell_i\}).
\end{equation*}
From now on, we may assume that there exists $A_i \in\bigcup_{F\in\mathcal{F}^*}\mathcal{I}(F)$ with $|A_i| = \ell_i$ for any $i\in [s-1]$.
By the property (ii), we have that for every $F\in\mathcal{F}^*$ there exists some $A\in\mathcal{I}(F)$ with $|A|=\ell_i$ for any $i\in [s-1]$. Given $i\in [s-1]$ and $F\in\mathcal{F}^*$,
let $\mathcal{A}_F^i=\{A\in\mathcal{I}(F):|A|=\ell_i\}$. Then $\mathcal{A}_F^i\not=\emptyset$.
By the property (vi), $\bigcup_{F\in\mathcal{F}^*}\mathcal{A}_{F}^i$ is an $n$-vertex $\{\ell_1,\ldots,\ell_{i-1}\}$-intersecting $\ell_i$-graph. Hence
\begin{equation*}
    \left|\bigcup_{F\in\mathcal{F}^*}\mathcal{A}_{F}^i\right|\le \Phi_{\ell_i}(n,\{\ell_1,\ldots,\ell_{i-1}\}).
\end{equation*}
For $A\in \bigcup_{F\in\mathcal{F}^*}\mathcal{A}_{F}^i$, let $\mathcal{F}_i^*[A]=\{F\in\mathcal{F}^*:A\subset F\}$ and $\mathcal{F}_i^*(A)=\{F\setminus A:F\in \mathcal{F}_i^*[A]\}$. Note that $A\in\binom{V(G)}{\ell_i}$. Let $N_A=\bigcap_{v\in A}N_{G}(v)$. Then  $G[N_A]$ is $t$-wise $(K_{r-\ell_i},\{0,\ell_{i+1}-\ell_i,\ell_{i+2}-\ell_i,\ldots,\ell_s-\ell_i\})$-intersecting on at most $n-\ell_i$ vertices. Thus
\begin{equation*}
    \begin{aligned}
        |\mathcal{F}_i^*[A]|
        &=|\mathcal{F}_i^{*}(A)|\\
        &\le N(K_{r-\ell_i},G[N_A])\\
        &\le \Psi_{r-\ell_i}(n-\ell_i,\{0,\ell_{i+1}-\ell_i,\ldots,\ell_s-\ell_i\},t).
    \end{aligned}
\end{equation*}
Since every $F\in\mathcal{F}^*$ contains at least one $\ell_i$-set in $\mathcal{A}_{F}^i$, we have
\begin{equation*}
    \begin{aligned}
        |\mathcal{F}^*|&\le \left|\bigcup_{A\in \bigcup_{F\in\mathcal{F}^*}\mathcal{A}_F^i}\mathcal{F}_i^*[A]\right|\\
        &\le\left|\bigcup_{F\in\mathcal{F}^*}\mathcal{A}_{F}^i\right|\Psi_{r-\ell_i}(n-\ell_i,\{0,\ell_{i+1}-\ell_i,\ell_{i+2}-\ell_i,\ldots,\ell_s-\ell_i\},t)\\
        &\le \Phi_{\ell_i}(n,\{\ell_1,\ldots,\ell_{i-1}\})\Psi_{r-\ell_i}(n-\ell_i,\{0,\ell_{i+1}-\ell_i,\ell_{i+2}-\ell_i,\ldots,\ell_s-\ell_i\},t).
    \end{aligned}
\end{equation*}
Together with $|\mathcal{F}|\le c^{-1}|\mathcal{F}^*|$, we are done.
\end{proof}
We still  need the following lemma.
\begin{lemma}[Zhao-Zhang~\cite{zhao2025countingcliquesprescribedintersection}]\label{t count}
    Let $r\ge 3$ be an integer. Let $0 = \ell_1 < \ell_2 < \ell_3 < r$ with $r - \ell_3 = \ell_3 - \ell_2 \neq \ell_2$. Then any $n$-vertex $\{\ell_1, \ell_2, \ell_3\}$-intersecting $r$-graph $\mathcal{H}$ satisfies
    \begin{equation*}
        |\mathcal{H}|\le \binom{n}{2}.
    \end{equation*}
\end{lemma}
\vskip.2cm
Now we are going to prove Theorem~\ref{thm 3}.
\vskip.2cm
\noindent
\emph{Proof of Theorem~\ref{thm 3}:}
Since $\ell_1,\ell_2,\ldots,\ell_s,r$ do not form an arithmetic progression, either $r-\ell_s\neq \ell_s-\ell_{s-1}$, or there is some $i\in [s-2]$ such that
\[
r-\ell_s= \ell_s-\ell_{s-1}=\cdots=\ell_{i+2}-\ell_{i+1}\neq \ell_{i+1}-\ell_i.
\]We consider two cases.

\textbf{Case 1:}  $r-\ell_s\neq \ell_s-\ell_{s-1}$.

In this case, $r-\ell_{s-1}\ge 3$. By Lemma \ref{lem 1}, $\Psi_{r-\ell_{s-1}}(n-\ell_{s-1},\{0,\ell_{s}-\ell_{s-1}\},t)=o(n^2)$. Applying Lemma \ref{key key lem} and Theorem \ref{deza}, we obtain, for large $n$,
\begin{equation*}
    \begin{aligned}
        &\Psi_{r}(n,L,t)\\
        &\le c^{-1}\max\bigl\{\Phi_{r}(n,L\setminus\{\ell_{s-1}\}),
                \Phi_{\ell_{s-1}}(n,\{\ell_1,\ldots,\ell_{s-2}\})
        \Psi_{r-\ell_{s-1}}(n-\ell_{s-1},\{0,\ell_{s}-\ell_{s-1}\},t)\bigr\}\\
        &\le c^{-1}\max\biggl\{\prod_{\ell\in L\setminus\{\ell_{s-1}\}}\frac{n-\ell}{r-\ell},
        \Phi_{\ell_{s-1}}(n,\{\ell_1,\ldots,\ell_{s-2}\}) o(n^2)\biggr\}\\
        &=c^{-1}\max\bigl\{O(n^{s-1}),
        \Phi_{\ell_{s-1}}(n,\{\ell_1,\ldots,\ell_{s-2}\})o(n^2)\bigr\}.
    \end{aligned}
\end{equation*}
    By Theorem~\ref{deza}, it is easy to derive that $\Phi_{\ell_{s-1}}(n,\{\ell_1,\ldots,\ell_{s-2}\})=O(n^{s-2})$.
    Here the case when $\ell_{s-1} < 3$ can be solved easily.
    Therefore $\Psi_{r}(n,L,t)=o(n^s)$.

    \textbf{Case 2:}  $r-\ell_s= \ell_s-\ell_{s-1}=\cdots=\ell_{i+2}-\ell_{i+1}\neq \ell_{i+1}-\ell_i$ for some $i\in [s-2]$. 

    By Lemma~\ref{t count}, we have
    \begin{equation*}
        \Phi_{\ell_{i+3}-\ell_i}(n,\{0,\ell_{i+1}-\ell_i,\ell_{i+2}-\ell_i\})\le \binom{n}{2},
    \end{equation*}
    where $\ell_{s+1}=r$.
    Combining Proposition \ref{linear pro}, we have
    \begin{equation}\label{impor eq}
        \Psi_{\ell_{i+3}-\ell_i}(n,\{0,\ell_{i+1}-\ell_i,\ell_{i+2}-\ell_i\},t)=O(n^2).
    \end{equation}
    Applying Lemma \ref{key key lem} and Theorem \ref{deza} again, we obtain, for large $n$,
    \begin{equation*}
        \begin{aligned}
            \Psi_{r}(n,L,t)
            &\le c^{-1}\max\bigl\{\Phi_{r}(n,L\setminus\{\ell_i\}),\\
            &\hspace{4em}
            \Phi_{\ell_i}(n,\{\ell_1,\ldots,\ell_{i-1}\})
            \Psi_{r-\ell_i}(n-\ell_i,\{0,\ell_{i+1}-\ell_i,\ldots,\ell_s-\ell_i\},t)\bigr\}\\
            &\le c^{-1}\max\bigl\{O(n^{s-1}),\\
            &\hspace{4em}
            \Phi_{\ell_i}(n,\{\ell_1,\ldots,\ell_{i-1}\})
            \Psi_{r-\ell_i}(n-\ell_i,\{0,\ell_{i+1}-\ell_i,\ldots,\ell_s-\ell_i\},t)\bigr\}.
        \end{aligned}
    \end{equation*}
    Since $\Phi_{\ell_i}(n,\{\ell_1,\ldots,\ell_{i-1}\})\le n^{i-1}$, it remains to show that
    \begin{equation*}
        \Psi_{r-\ell_i}(n-\ell_i,\{0,\ell_{i+1}-\ell_i,\ldots,\ell_s-\ell_i\},t)
        =O(n^{s-i}).
    \end{equation*}

    When $i=s-2$, it follows from \eqref{impor eq}.
    When $i\le s-3$, let $L'=\{0,\ell_{i+1}-\ell_i,\ldots,\ell_s-\ell_i\}$. By Lemma \ref{key key lem}, Theorem \ref{deza}, and \eqref{impor eq}, we have
    \begin{equation*}
        \begin{aligned}
            &\Psi_{r-\ell_i}(n-\ell_i,\{0,\ell_{i+1}-\ell_i,\ldots,\ell_s-\ell_i\},t)\\
            \le & c_{i}^{-1}\max\bigl\{
            \Phi_{r-\ell_i}(n-\ell_i,L'\setminus\{\ell_{i+3}-\ell_i\}),\\
            &\hspace{5em}
            \Phi_{\ell_{i+3}-\ell_i}(n-\ell_i,\{0,\ell_{i+1}-\ell_i,\ell_{i+2}-\ell_i\})\\
            &\hspace{5em}\cdot
            \Psi_{r-\ell_{i+3}}(n-\ell_{i+3},\{0,\ell_{i+4}-\ell_{i+3},\ldots,\ell_{s}-\ell_{i+3}\},t)
            \bigr\}\\
            \le & c_{i}^{-1}\max\bigl\{O(n^{s-i}),\\
            &\hspace{5em}
            n^2\cdot
            \Psi_{r-\ell_{i+3}}(n-\ell_{i+3},\{0,\ell_{i+4}-\ell_{i+3},\ldots,\ell_{s}-\ell_{i+3}\},t)
            \bigr\}.
        \end{aligned}
    \end{equation*}
    By Proposition \ref{linear pro} and Theorem \ref{deza}~(the case when $r-\ell_{i+3}<3$ can be solved easily), we have
    \begin{equation*}
         \Psi_{r-\ell_{i+3}}(n-\ell_{i+3},\{0,\ell_{i+4}-\ell_{i+3},\ldots,\ell_{s}-\ell_{i+3}\},t)
         =O(n^{s-i-2}).
    \end{equation*}
    Hence $\Psi_r(n,L,t)=O(n^{s-1})$ and we are done.
\hfill $\square$ \par

\section{Proof of Theorem~\ref{thm 4}}\label{sec: proof AP}
In this section, assume that $L =\{\ell_1,\ell_2,\ldots,\ell_s\}\subseteq[0,r-1]$, $s\not\in \{1,r\}$, and $\ell_1,\ell_2,\ldots,\ell_s,r$ form an arithmetic progression with common difference $d$.
For convenience, let $\ell_{s+1} = r$.
The lower bound follows from the following construction, which is taken from~\cite{zhao2025countingcliquesprescribedintersection}.
\begin{enumerate}
    \item Start with the Tur\'an graph $T(\lfloor(n-\ell_1)/d\rfloor, s)$.
    \item Replace each vertex by a clique of size $d$.
    \item Add a clique of size $\ell_1$ and make it complete joint to all remaining vertices.
\end{enumerate}
This graph is $t$-wise $(K_r,L)$-intersecting and contains $(1+o(1))\left(\frac{n-\ell_1}{r-\ell_1}\right)^{s}$ copies of $K_r$.

    We now prove the upper bound. First consider the case when $d\ge 2$. Let $G$ be an extremal $t$-wise $(K_r,L)$-intersecting graph on $n$ vertices. By the lower-bound construction above, $N(K_r,G)=\Omega(n^s)$. Therefore, Theorem~\ref{thm: t L} implies that $|\bigcap_{S\in \mh_{G}^{r}}S|\geq \ell_1$.
    Fix a common $K_{\ell_1}$ in $\bigcap_{S\in \mh_{G}^{r}}S$ and restrict to its common neighborhood. It suffices to prove the upper bound for $\Psi_{r-\ell_1}(n-\ell_1,L_{d,s},t)$, where $L_{d,s}=\{0,d,2d,\ldots,(s-1)d\}$.
    Thus we may assume from now on that $G$ is an extremal $t$-wise $(K_r,L_{d,s})$-intersecting graph on $n$ vertices.

    We call a subset $S\subseteq V(G)$ an {\em atom} if the following conditions hold:
    \begin{enumerate}
        \item $|S|\geq d$,
        \item for any $A \in \mathcal{H}_G^{r}$, either $S \subseteq A$ or $S \cap A = \emptyset$, and
        \item $S$ is inclusion-maximal with respect to these properties.
    \end{enumerate}

    \begin{claim}\label{claim: atom dj}
    For any two atoms $S,S'$, we have $S\cap S'=\emptyset$.
    \end{claim}
    \noindent
    \textbf{Proof of Claim~\ref{claim: atom dj} }
        Suppose, to the contrary, that there are two atoms $S,S'$ satisfy $S\cap S'\neq \emptyset$.
        For any $A\in \mh_{G}^{r}$, if $A\cap S=\emptyset$, then $A\cap S'=\emptyset$ because $S'$ is an atom.
        If $S\subseteq A$, then $S'\cap A\neq \emptyset$, and hence $S'\subseteq A$.
        Thus $S\cup S'$ is also an atom, a contradiction with the inclusion-maximality.\q

Let $\mathcal{S}=\{S\in\binom{V(G)}{d}:\text{$S$ is an atom}\}$, and define $X_1=\bigcup_{S\in\mathcal{S}}S\subseteq V(G)$. Note that $r = sd$.
We define an auxiliary graph $G'$ with $V(G')=\mathcal{S}$, whose edge set is
\begin{equation*}
    E(G')=\left\{\{S_1,S_2\}\in\binom{\mathcal{S}}{2}: \text{there exists $A\in\mh_{G}^{sd}$ such that $S_1\cup S_2\subseteq A$}\right\}.
\end{equation*}Let $K_{s;t}$ denote the complete $(s+1)$-partite graph with $s$ parts of size $1$ and one part of size $t$.
\begin{claim}\label{claim: K_ free}
 $G'$ is $K_{s;t}$-free.
\end{claim}
\noindent
    \textbf{Proof of Claim~\ref{claim: K_ free} }
    Otherwise, a copy of $K_{s;t}$ in $G'$ corresponds to a copy of  $K_{sd}+t K_{d}$ in $G$ by Claim \ref{claim: atom dj}. Then $G$  contains $K_{sd-1}+\overline{K_t}$ as a subgraph, where $\overline{K_t}$ is the  ‌complement of $K_t$. This yields $t$ distinct copies of $K_{sd}$ in $G$ whose intersection has size $sd-1 \notin L$, a contradiction.\q

Let $X_{0}=V(G)\setminus X_1$.
\begin{claim}\label{claim: A contain x is small}
For any $x\in X_0$, we have $|\{A: x\in A\in\mh_{G}^{sd}\}|=o(n^{s-1})$.
\end{claim}
\noindent
    \textbf{Proof of Claim~\ref{claim: A contain x is small} }
    Fix $x \in X_0$.
Recall the link hypergraph of $x$, $\mh_{G}^{sd}(x)$,  where $\mh_{G}^{sd}$ has hyperedge set $\{A\setminus\{x\}: x\in A\in\mh_{G}^{sd}\}$.
Since $G$ is $t$-wise $(K_{sd},L_{d,s})$-intersecting, $\mh_{G}^{sd}(x)$ is $t$-wise $\{d-1,2d-1,\ldots,(s-1)d-1\}$-intersecting.
Let $I=\bigcap_{A \in\mh_{G}^{sd}(x)}A$.
If $|I|\leq d-2$, then Theorem \ref{thm: t L} gives $|\mh_{G}^{sd}(x)|=o(n^{s-1})$, and we are done.
If $|I|\geq d$, then $\mh_{G}^{sd}(x)$ is $t$-wise $\{2d-1,\ldots,(s-1)d-1\}$-intersecting, and Theorem \ref{thm: t L} gives $|\mh_{G}^{sd}(x)|=O(n^{s-2})=o(n^{s-1})$.
The only remaining case is $|I|=d-1$. By $d\ge 2$, $I\not=\emptyset$.

Let $Q=\{x\}\cup I$. Then $A\cap Q\not=\emptyset$ for any $A \in\mh_{G}^{sd}(x)$. Since $x\notin X_1$, we have $Q\notin\mathcal{S}$. We claim that for any $P\in \mathcal{S}$, $Q\not\subseteq P$. Suppose there is $P\in \mathcal{S}$ such that $Q\subseteq P$. Then $P\cap A\not=\emptyset$ for any $A \in\mh_{G}^{sd}(x)$. So $P\cap (A\cup \{x\})\not=\emptyset$ for any $A \in\mh_{G}^{sd}(x)$. Since $P$ is an atom, we have $P\subseteq A\cup \{x\}$ for any $A \in\mh_{G}^{sd}(x)$. Hence we have $P\subseteq I\cup \{x\}$ which implies $P=Q$, a contradiction with $Q\notin\mathcal{S}$.

 Therefore there is $D \in\mh_{G}^{sd}$ such that $Q\cap D\not=\emptyset$ and $Q\not\subseteq D$. Since $Q\not\subseteq D$, we have $x\not\in D$. Since $Q\cap D\not=\emptyset$,  $D\cap I\neq \emptyset$.
For any $B\in \mh_{G}^{sd}(x)$, we claim $B\cap (D\setminus I)\neq \emptyset$. Suppose there is $B\in \mh_{G}^{sd}(x)$ such that $B\cap (D\setminus I)= \emptyset$.
Then we can choose  $B_1,B_2,\ldots,B_{t-2}\in \mh_{G}^{sd}(x)\setminus\{B\}$ (otherwisw the claim holds immediate).
Then $0<|D\cap (B\cup\{x\})\cap \cap _{i=1}^{t-2}(B_i \cup \{x\})|\le d-1$, a contradiction with  $G$ being  $t$-wise $(K_r,L_{d,s})$-intersecting.

Hence $|\mh_{G}^{sd}(x)|\leq\sum_{y\in D\setminus I} |\{A \in\mh_{G}^{sd}:~ x,y \in A\}|$.
Since every hyperedge in $\{A \in\mh_{G}^{sd}:~ x,y \in A\}$ contains $x,y$ and $I$,  this family is $\{2d, 3d, \ldots, (s-1)d\}$-intersecting.
By Theorem~\ref{thm: t L}, we obtain $|\mh_{G}^{sd}(x)| = O(n^{s-2}) = o(n^{s-1})$.\q

Now we complete the proof of Theorem~\ref{thm 4}. Write $n=md+\lambda$ with $0\le \lambda<d$. Recall that $|X_1|=|\mathcal{S}|d$ and $|X_0|=|V(G)|-|X_1|$. Then $|\mathcal{S}|d=|X_1|\le n-\lambda$. For any $A\in \mh_{G}^{sd}\setminus \mh_{G[X_1]}^{sd}$, we have $A\cap X_0\neq\emptyset$. By Claim \ref{claim: A contain x is small}, we have
\begin{equation*}
    \begin{aligned}
        |\mh_{G}^{sd}|&=|\mh_{G}^{sd}\setminus \mh_{G[X_1]}^{sd}|+|\mh_{G[X_1]}^{sd}|\\
        & \leq \sum_{y\in X_0}|\mh_{G}^{sd}(y)|+N(K_{sd},G[X_1])\\
        & \leq o(n^s) + N(K_s,G')\\
        & \le o(n^s)+(1+o(1)) \left(\frac{|\mathcal{S}|}{s}\right)^{s}\\
        & \le o(n^s)+(1+o(1)) \left(\frac{n-\lambda}{sd}\right)^{s}\\
        &= (1+o(1))\left(\frac{n}{sd}\right)^s.
    \end{aligned}
\end{equation*}
The estimate for $N(K_s,G')$ follows from Claim \ref{claim: K_ free} and Theorem~\ref{alon pro}.

    It remains to consider the case $d=1$. Since $s\not\in \{1,r\}$, we have $\ell_1\ge 1$.
    Let $G$ be an extremal $t$-wise $(K_r,L)$-intersecting graph on $n$ vertices. By Theorem~\ref{thm: t L} and the lower-bound construction above, we have $|\bigcap_{S\in \mh_{G}^{r}}S|\geq \ell_1$. Fix a common $K_{\ell_1}$ in $\bigcap_{S\in \mh_{G}^{r}}S$ and denote it by $R$.
    Let $N=\bigcap_{v\in V(R)}N_{G}(v)$. Then $|N|\le n-\ell_1$.

We claim that $G[N]$ is $tK_{r-\ell_1+1}$-free.
Suppose otherwise, it would correspond to a copy of $K_{\ell_1}+t K_{r-\ell_1+1}$ in $G$, which leads to $t$ copies of $K_r$ whose intersection has size $\ell_1-1 \notin L$, a contradiction.
Note that $\chi(tK_{r-\ell_1+1})=r-\ell_1+1=s+1$.
By Theorem~\ref{alon pro}, we have
\begin{equation*}
    \begin{aligned}
        N(K_r,G)&=N(K_{r-\ell_1},G[N])
\le(1+o(1))\binom{s+1-1}{r-\ell_1}\left(\frac{n-\ell_1}{s+1-1}\right)^{r-\ell_1}\\
        &=(1+o(1))\left(\frac{n-\ell_1}{r-\ell_1}\right)^{s}.
    \end{aligned}
\end{equation*}
This completes the proof.

\section*{Declaration of competing interest}
The authors declare that they have no known competing financial interests or personal relationships that could have appeared to influence the work reported in this paper.

\section*{Data availability}
No data was used for the research described in the article.

\section*{Declaration on AI Use}
The authors declare that no artificial intelligence tools were used in the development of the mathematical arguments, proofs, or the selection and verification of citations in this article.

\end{document}